\documentclass[12pt,american]{article}
\usepackage[T1]{fontenc}
\usepackage[latin9]{inputenc}
\usepackage{color}
\usepackage{babel}
\usepackage{textcomp}
\usepackage{mathrsfs}
\usepackage{bm}
\usepackage{amsthm}
\usepackage{amsmath}
\usepackage{amssymb}
\usepackage{graphicx}
\usepackage{esint}
\usepackage[unicode=true,pdfusetitle,
 bookmarks=true,bookmarksnumbered=false,bookmarksopen=false,
 breaklinks=false,pdfborder={0 0 1},backref=page,colorlinks=true]
 {hyperref}
\hypersetup{
 linkcolor=red, citecolor=blue, urlcolor=blue, filecolor=blue}

\makeatletter
\theoremstyle{plain}
\newtheorem{thm}{\protect\theoremname}
  \theoremstyle{definition}
  \newtheorem{defn}[thm]{\protect\definitionname}
  \theoremstyle{plain}
  \newtheorem{cor}[thm]{\protect\corollaryname}
  \theoremstyle{plain}
  \newtheorem{prop}[thm]{\protect\propositionname}
  \theoremstyle{remark}
  \newtheorem{rem}[thm]{\protect\remarkname}

\usepackage{datetime,color,currfile,amssymb,amsmath,pdfsync,hyperref}
\DeclareMathAlphabet{\mathitbf}{OML}{cmm}{b}{it}

\makeatother

  \providecommand{\corollaryname}{Corollary}
  \providecommand{\definitionname}{Definition}
  \providecommand{\propositionname}{Proposition}
  \providecommand{\remarkname}{Remark}
\providecommand{\theoremname}{Theorem}

\begin{document}

\title{Local times for multifractional Brownian motion in higher dimensions:
A white noise approach}

\author{\textbf{Wolfgang Bock} ,\\
CMAF, Universidade de Lisboa,\\
1649-003 Lisbon, Portugal.\\
Email: bock@campus.ul.pt\and\textbf{ José Luís da Silva},\\
CCM, University of Madeira, Campus da Penteada,\\
9020-105 Funchal, Portugal.\\
Email: luis@uma.pt\and\textbf{Herry P.}~\textbf{Suryawan}\\
Department of Mathematics\\
Sanata Dharma University\\
Yogyakarta, Indonesia\\
Email: herrypribs@usd.ac.id}
\maketitle
\begin{abstract}
We present the expansion of the multifractional Brownian (mBm) local
time in higher dimensions, in terms of Wick powers of white noises
(or multiple Wiener integrals). If a suitable number of kernels is
subtracted, they exist in the sense of generalized white noise functionals.
Moreover we show the convergence of the regularized truncated local
times for mBm in the sense of Hida distributions. \medskip{}

\end{abstract}

\section{Introduction}

Over the last decades fractional Brownian motion (fBm) with Hurst
parameter $H$ has become an intensively studied object. This centered
Gaussian process $B_{H}$ with covariance function

\[
\mathbb{E}(B_{H}(t)B_{H}(s))=\frac{{1}}{2}\left(|t|^{2H}+|s|^{2H}-|t-s|^{2H}\right),\quad t,s>0,
\]
was first introduced by Mandelbrot and Van-Ness \cite{MandelbrotNess1968}.
Instead of giving an exhaustive overview about fBm we refer to the
articles \cite{Alos-Nualar-03,Bender03a,Bender03,DHP00,Elliott-Hoek-03,HO03,Nualart2006,drumond-oliveira-silva08}
and monographs \cite{oks,mishura08} and the references therein. 

Due to its properties such as Hölder continuity of any order less
than $H$, long-range dependence and stationary increments, the process
is used for modeling problems from telecommunications, finance and
engineering. Although there are various problems accessible, the use
of fBm involves a restriction to a certain Hölder continuity $H$
of the paths for all times of the process. For many applications this
is too restrictive and variable time-dependent Hölder continuities
of the paths are needed. 

To overcome this Lévy Véhel and Peltier \cite{Peltier:1995ug} and
Benassi et.al. \cite{Benassi:1997jf} independently introduced multifractional
Brownian motion (mBm) $B_{h}$, where the regularity of the paths
is a function of time. The covariance of the centered Gaussian process
$B_{h}$ is given by

\[
\mathbb{E}(B_{h}(t)B_{h}(s))=\frac{C\left(\frac{h(t)+h(s)}{2}\right)^{2}}{C(h(t))C(h(s))}\left[\frac{1}{2}(t^{h(t)+h(s)}+s^{h(t)+h(s)}-|t-s|^{h(t)+h(s)})\right],
\]
where $h:[0,T]\longrightarrow(1/2,1)$ is a continuous function and 

\[
C(x):=\left(\frac{2\pi}{\Gamma(2x+1)\sin(\pi x)}\right)^{1/2},
\]
where $\Gamma$ is the Gamma function. Different properties of this
process are recently studied, such as Hölder continuity of the paths
and Hausdorff dimension \cite{Benassi:1997jf,Boufoussi:2008}, as
well as local times of mBm \cite{ayache2011,boufoussi2007,meerschaert2008}
and estimates for the local Hurst parameters \cite{BFG13}. In white
noise analysis mBm was treated recently in \cite{Lebovits:2014es}
together with its respective stochastic calculus.

In this article we use a white noise approach to determine the kernels
in the Wiener-Itô-Segal chaos decomposition of the (truncated) local
time of a $d$-dimensional mBm for $h:[0,T]\longrightarrow(1/2,1)$.
As in \cite{drumond-oliveira-silva08} for fBm, we show the convergence
of the regularized local time for mBm to the truncated local time
via the convergence of Hida distributions.

\section{Gaussian white noise analysis}

In this section we review some of the standard concepts and theorems
of white noise analysis used throughout this work, and refer to \cite{HKPS93,KLPSW96,Kuo96}
and references therein for a detailed presentation. 

We start with the basic Gel'fand triple 
\[
S_{d}\subset L_{d}^{2}\subset S'_{d},
\]
where $S_{d}:=S(\mathbb{R},\mathbb{R}^{d})$, $d\in\mathbb{N}$, is
the space of vector valued Schwartz test functions, $S'_{d}$ its
topological dual and the central Hilbert space $L_{d}^{2}:=L^{2}(\mathbb{R},\mathbb{R}^{d})$
of square integrable vector valued functions, i.e., 
\[
|f|_{0}^{2}=\sum_{i=1}^{d}\int_{\mathbb{R}}f_{i}^{2}(x)\, dx,\quad f\in L_{d}^{2}.
\]

Since $ $$S_{d}$ is a nuclear space, represented as projective limit
of a decreasing chain of Hilbert spaces $ $$(H_{p})_{p\in\mathbb{{N}}}$,
see e.g.\cite{RS75a} and \cite{GV68}, i.e. 
\[
S_{d}=\bigcap_{p\in\mathbb{{N}}}H_{p},
\]
 we have that $S_{d}$ is a countably Hilbert space in the sense of
Gel'fand and Vilenkin \cite{GV68}. We denote the corresponding norm
on $H_{p}$ by $|\cdot|_{p},$ with the convention $H_{0}=L_{d}^{2}$.
Let $H_{-p}$ be the dual space of $H_{p}$ and let $\langle\cdot,\cdot\rangle$
denote the dual pairing on $H_{-p}\times H_{p}$. $H_{p}$ is continuously
embedded into $L_{d}^{2}$. By identifying $L_{d}^{2}$ with its dual
via the Riesz isomorphism, we obtain the chain $H_{p}\subset L_{d}^{2}\subset H_{-p}$.
Note that $S'_{d}=\bigcup_{p\in\mathbb{{N}}}H_{-p}$, i.e. $S'_{d}$
is the inductive limit of the increasing chain of Hilbert spaces $(H_{-p})_{p\in\mathbb{{N}}}$,
see e.g.~\cite{GV68}. We denote the dual pairing of $\ensuremath{S_{d}'}$
and $\ensuremath{S_{d}}$ also by $\langle\cdot,\cdot\rangle$.

Let $\mathscr{B}$ be the $\sigma$-algebra generated by cylinder
sets on $S'_{d}$. By Minlos\textquoteright{} theorem there is a unique
probability measure $\mu_{d}$ on $(S'_{d},\mathscr{B})$ with characteristic
function given by
\[
\int_{S'_{d}}e^{i\langle\bm{w},\bm{\varphi}\rangle}d\mu_{d}(\bm{w})=\exp\left(-\frac{1}{2}|\bm{\varphi}|_{0}^{2}\right),\quad\bm{\varphi}\in S_{d}.
\]
Hence, we have defined the white noise measure space $(S'_{d},\mathscr{B},\mu_{d})$.
The complex Hilbert space $L^{2}(\mu_{d}):=L^{2}(S'_{d},\mathscr{B},\mu_{d})$
is canonically isomorphic to the Fock space of symmetric square integrable
functions 
\begin{equation}
L^{2}(\mu_{d})\simeq\left(\bigoplus_{k=0}^{\infty}\mathrm{Sym}L^{2}(\mathbb{R}^{k},k!d^{k}x)\right)^{\otimes d}\label{eq:Fock-isomorphism}
\end{equation}
which implies the Wiener-Itô-Segal chaos decomposition for any element
$F$ in $L^{2}(\mu_{d})$
\[
F(\bm{w})=\sum_{\bm{n}\in\mathbb{N}^{d}}\langle:\bm{w}^{\otimes\bm{n}}:,F_{\bm{n}}\rangle
\]
with the kernel function $F_{\bm{n}}$ in the Fock space. We introduce
the following notation for simplicity
\[
\bm{n}=(n_{1},\ldots,n_{d})\in\mathbb{N}_{0}^{d},\quad n=n_{1}+\ldots+n_{d},\quad\bm{n}!=n_{1}!\ldots n_{d}!
\]
and for any $\bm{w}=(w_{1},\ldots,w_{d})\in S'_{d}$
\[
:\bm{w}^{\otimes\bm{n}}:=:w_{1}^{\otimes n_{1}}:\otimes\ldots\otimes:w_{d}^{\otimes n_{d}}:,
\]
where $:w^{\otimes n}:$ denotes the $n$-th Wick tensor power of
the element $w\in S_{1}'$, for its definition see e.g. \cite{HKPS93}.
For any $F\in L^{2}(\mu_{d})$, the isomorphism (\ref{eq:Fock-isomorphism})
yields 
\[
\|F\|_{L^{2}(\mu_{d})}^{2}:=\sum_{\bm{n}\in\mathbb{N}^{d}}\bm{n}!|F_{\bm{n}}|_{0}^{2},
\]
where the symbol $|\cdot|_{0}$ is also preserved for the norms on
$L^{2}(\mathbb{{R}},\mathbb{{R}}^{d})_{\mathbb{{C}}}^{\hat{\otimes}\bm{{n}}}$,
for simplicity. By the standard construction with the space of square-integrable
functions w.r.t.~$\mu_{d}$ as central space, we obtain the Gel'fand
triple of Hida test functions and Hida distributions.
\[
(S_{d})\subset L^{2}(\mu_{d})\subset(S_{d})'.
\]
In the following we denote the dual pairing between elements of $(S_{d})'$
and $(S_{d})$ by $\langle\!\langle\cdot,\cdot\rangle\!\rangle$.
For $F\in L^{2}(\mu_{d})$ and $\varphi\in(S_{d})$, with kernel functions
$f_{\bm{n}}$ and $\varphi_{\bm{n}}$, resp. the dual pairing yields
\[
\langle\!\langle F,\varphi\rangle\!\rangle=\sum_{\bm{n}}\bm{n}!\langle f_{\bm{n}},\varphi_{\bm{n}}\rangle
\]
 This relation extends the chaos expansion to $\Phi\in(S_{d})'$ with
distribution valued kernels $\Phi_{\bm{n}}$ such that 
\[
\langle\!\langle\Phi,\varphi\rangle\!\rangle=\sum_{\bm{n}}\bm{n}!\langle\Phi_{\bm{n}},\varphi_{\bm{n}}\rangle,
\]
 for every generalized test function $\varphi\in(S_{d})$ with kernels
$\varphi_{\bm{n}}$.

Instead of reviewing the detailed construction of these spaces we
give a characterization in terms of the $S$-transform.
\begin{defn}
Let $\boldsymbol{\xi}\in S_{d}$, then $:\exp(\langle.,\boldsymbol{\xi}\rangle):\,\,:=\sum\limits _{k=0}^{\infty}\frac{{1}}{k!}\langle:.^{\otimes n}:,\boldsymbol{\xi^{\otimes n}}\rangle\in(S_{d})$
and we define the $S$-transform of $\Phi\in(S_{d})'$ by

\[
(S\Phi)(\bm{\xi})=\langle\!\langle\Phi,:\exp(\langle.,\boldsymbol{\xi}\rangle):\rangle\!\rangle.
\]
 
\end{defn}

\begin{defn}
[$U$-functional]\label{def:U-functional}A function $F:S_{d}\longrightarrow\mathbb{C}$
is called a $U$-functional whenever
\begin{enumerate}
\item for every $\bm{\varphi}_{1},\bm{\varphi}_{2}\in S_{d}$ the mapping
$\mathbb{R}\ni\lambda\mapsto F(\lambda\bm{\varphi}_{1}+\bm{\varphi}_{2})$
has an entire extension to $z\in\mathbb{C}$,
\item there are constants $K_{1},K_{2}>0$ such that 
\[
\left|F(z\bm{\varphi})\right|\leq K_{1}\exp\left(K_{2}\left|z\right|^{2}\|\bm{\varphi}\|^{2}\right),\quad\forall z\in\mathbb{C},\bm{\varphi}\in S_{d}
\]
 for some continuous norm $\|\cdot\|$ on $S_{d}$.
\end{enumerate}
\end{defn}
We are now ready to state the aforementioned characterization result.
\begin{thm}
[cf.~\cite{KLPSW96}, \cite{PS91}]\label{thm:S-transform-U-functional}
The $S$-transform defines a bijection between the space $(S_{d})'$
and the space of $U$-functionals. In other words, $\Phi\in(S_{d})'$
if and only if $S\Phi:S_{d}\to\mathbb{{C}}$ is a $U-$functional.
\end{thm}
Based on Theorem \ref{thm:S-transform-U-functional} a deeper analysis
of the space $(S_{d})'$ can be done. The following Corollaries concern
the convergence of sequences and the Bochner integration of families
of generalized functions in $(S_{d})'$ (for more details and proofs
see e.g.~\cite{HKPS93}, \cite{KLPSW96}, \cite{PS91}).
\begin{cor}
\label{cor:conv-hida-seq}Let $(\Phi_{n})_{n\in\mathbb{N}}$ be a
sequence in $(S_{d})'$ such that
\begin{enumerate}
\item for all $\bm{\varphi}\in S_{d}$, $((S\Phi_{n})(\bm{\varphi}))_{n\in\mathbb{N}}$
is a Cauchy sequence in $\mathbb{C}$,
\item there are $K_{1},K_{2}>0$ such that for some continuous norm $\|\cdot\|$
on $S_{d}$ one has
\[
|(S\Phi_{n})(z\bm{\varphi})|\leq K_{1}\exp\left(K_{2}|z|^{2}\|\bm{\varphi}\|^{2}\right),\quad\bm{\varphi}\in S_{d},\; n\in\mathbb{N},\,\, z\in\mathbb{{C}}.
\]

\end{enumerate}

Then $(\Phi_{n})_{n\in\mathbb{N}}$ converges strongly in $(S_{d})'$
to a unique Hida distribution.

\end{cor}

\begin{cor}
\label{cor:hida_integral}Let $(\Omega,\mathcal{B},m)$ be a measure
space and $\lambda\mapsto\Phi_{\lambda}$ be a mapping from $\Omega$
to $(S_{d})'$. We assume that the $S$-transform of $\Phi_{\lambda}$
fulfills the following two properties:
\begin{enumerate}
\item The mapping $\lambda\mapsto(S\Phi_{\lambda})(\mathbf{\text{\ensuremath{\bm{{\varphi}}}}})$
is measurable for every $\bm{{\varphi}}\in S_{d}$, 
\item The $S\Phi_{\lambda}$ obeys the estimate 
\[
|(S\Phi_{\lambda})(z\bm{\varphi})|\leq C_{1}(\lambda)\exp\left(C_{2}(\lambda)|z|^{2}\|\bm{\varphi}\Vert^{2}\right),\quad z\in\mathbb{C},\bm{\varphi}\in S_{d},
\]
 for some continuous norm $\Vert\cdot\Vert$ on $S_{d}$ and for $C_{1}\in L^{1}(\Omega,m)$,
$C_{2}\in L^{\infty}(\Omega,m)$. 
\end{enumerate}

Then 
\[
\int_{\Omega}\Phi_{\lambda}\, dm(\lambda)\in(S_{d})'
\]
 and 
\[
S\left(\int_{\Omega}\Phi_{\lambda}\, dm(\lambda)\right)(\bm{\varphi})=\int_{\Omega}(S\Phi_{\lambda})(\bm{\varphi})\, dm(\lambda),\quad\bm{\varphi}\in S_{d}.
\]

\end{cor}
At the end of this section we introduce the notion of truncated kernels,
defined via their Wiener-Itô-Segal chaos decomposition.
\begin{defn}
\label{def:trunc_kern}For $\Phi\in(S_{d})'$ with kernels ${F_{\bm{n}}},\bm{{n}}\in\mathbb{{N}}_{0}^{d}$
and $k\in\mathbb{N}_{0}$ we define truncated Hida distribution by
\[
\Phi^{(k)}=\sum_{\bm{n}\in\mathbb{N}_{0}^{d}:n\geq k}\langle:\cdot^{\otimes\bm{n}}:,F_{\bm{n}}\rangle.
\]
 Obviously one has $\Phi^{(k)}\in(S_{d})'.$
\end{defn}

\section{Multifractional Brownian motion}

Multifractional Brownian motion (mBm) in dimension $1$, was introduced
by Peltier and Lévy Véhel \cite{Peltier:1995ug}, Benassi et al.~\cite{Benassi:1997jf}
as zero mean Gaussian process with covariance given by
\[
R_{h}(t,s)=\frac{C\left(\frac{h(t)+h(s)}{2}\right)^{2}}{C(h(t))C(h(s))}\left[\frac{1}{2}(t^{h(t)+h(s)}+s^{h(t)+h(s)}-|t-s|^{h(t)+h(s)})\right],
\]
where the normalizing constant is defined by 
\[
C(x):=\left(\frac{2\pi}{\Gamma(2x+1)\sin(\pi x)}\right)^{1/2},
\]
where $\Gamma$ is the Gamma function. Let us fix some notations.
By $\hat{u}$ we denote the Fourier transform of $u$ and let $L_{\mathrm{loc}}^{1}(\mathbb{R})$
be the set of measurable functions which are locally integrable in
$\mathbb{R}$. Each $f\in L_{\mathrm{loc}}^{1}(\mathbb{R})$ gives
rise to an element in $S'_{1}$, denoted by $T_{f}$ namely, for any
$\varphi\in S_{1}$, $\langle T_{f},\varphi\rangle=\int_{\mathbb{R}}f(x)\varphi(x)\, dx$. 

In order to obtain a realization of mBm in the framework of white
noise analysis we introduce the following operator, see \cite{Lebovits:2014es}.
For any $H\in(1/2,1)$ we define the operator 
\[
(\widehat{M_{H}u})(y)=\frac{\sqrt{2\pi}}{C(H)}|y|^{1/2-H}\hat{u}(y),
\]
where $\hat{u}$ denotes the Fourier transform of the function $u$.
The operator $M_{H}$ is well defined in the space 
\[
L_{H}^{2}(\mathbb{R}):=\{u\in S_{1}':\;\hat{u}=T_{f};\; f\in L_{\mathrm{loc}}^{1}(\mathbb{R})\;\mathrm{and}\;\|u\|_{H}<\infty\},
\]
where the norm 
\[
\|u\|_{H}^{2}:=\frac{1}{C(H)^{2}}\int_{\mathbb{R}}|x|^{1-2H}|\hat{u}(x)|^{2}\, dx
\]
is the inner product norm on the Hilbert space $L_{H}^{2}(\mathbb{R})$
given by
\[
(u,v)_{H}=\frac{1}{C(H)^{2}}\int_{\mathbb{R}}|x|^{1-2H}\hat{u}(x)\overline{\hat{v}}(x)\, dx.
\]
Another possible representation for the operator $M_{H}$ is as follows
\cite[Eq.~2.4]{Lebovits:2014es}
\begin{equation}
(M_{H}\varphi)(x)=\gamma(H)\big<|\cdot|^{H-3/2},\varphi(x+\cdot)\rangle=:\gamma(H)\big<\Theta_{H},\varphi(x+\cdot)\rangle,\quad\varphi\in S_{1}(\mathbb{R}),\label{eq:MH_gf}
\end{equation}
where 
\[
\gamma(H):=\frac{\sqrt{\Gamma(2H+1)\sin(\pi H)}}{2\Gamma(H-1/2)\cos(\pi(H-1/2)/2)}.
\]
Note that $\Theta_{H}$ is a generalized function from $S'_{1}$,
i.e. for any $\psi\in S_{1}$ we have $\langle\Theta_{H},\psi\rangle=\langle|y|^{H-3/2},\psi(y)\rangle.$ 

The operator $M_{H}$ establishes an isometry between the Hilbert
spaces $L_{H}^{2}(\mathbb{R})$ and $L^{2}(\mathbb{R})$ \cite[Prop.~2.10]{Lebovits:2014es}.
Below we review some useful properties of the operator $M_{H}$. 
\begin{prop}
\label{prop:Properties_M_H}Let $M_{H}$ be the operator defined above
and $H\in(1/2,1)$.
\begin{enumerate}
\item Then $M_{H}$ is an isometric isomorphism between the Hilbert spaces
$L_{H}^{2}(\mathbb{R})$ and $L^{2}(\mathbb{R})$.
\item For any $f,g\in L^{2}(\mathbb{R})\cap L_{H}^{2}(\mathbb{R})$ we have
\[
\int_{\mathbb{R}}f(x)(M_{H}g)(x)\, dx=\int_{\mathbb{R}}(M_{H}f)(x)g(x)\, dx.
\]
Moreover, for any $f\in L_{\mathrm{loc}}^{1}(\mathbb{R})\cap L_{H}^{2}(\mathbb{R})$
and $g\in S_{1}$, we have
\[
\langle f,M_{H}g\rangle=(M_{H}f,g)_{L^{2}(\mathbb{R})}.
\]

\item There exists a constant $D$ such that for every $k\in\mathbb{N}_{0}:=\mathbb{N}\cup\{0\}$
we have
\[
\max_{x\in\mathbb{R}}|(M_{H}e_{k})(x)|\leq\frac{D}{C(H)}(k+1)^{2/3},
\]
where $e_{k}:=M_{H}^{-1}h_{k}$ and $h_{k}$ is the k-th Hermite function. 
\item There exists $p\in\mathbb{N}$ such that for all $t\in[0,T]$ and
$\varphi\in S_{1}$ we have
\[
\left|\int_{\mathbb{R}}\varphi(x)(M_{H}1\!\!1_{[0,t)})(x)\, dx\right|\leq|\gamma(H)|\,|\Theta_{H}|_{-p}\, t\,\sup_{x\in\mathbb{R}}|\varphi(x+\cdot)|_{p},
\]
with $\Theta_{H}$ as in (\ref{eq:MH_gf}).
\end{enumerate}
\end{prop}
\begin{proof}
The items 1., 2.~and 3.~are proved in \cite[Thm~2.14, 2.15]{Lebovits:2014es}.
Assertion $4.$~is a direct consequence of 2.~under the use of the
Cauchy-Schwarz inequality and the representation (\ref{eq:MH_gf})
for the operator $M_{H}$.
\end{proof}
Next we define the operator $M_{H}$ for a measurable functional parameter
$h:[0,T]\longrightarrow(1/2,1)$. For two indicator functions we define
\[
R_{h}(t,s):=(1\!\!1_{[0,t)},1\!\!1_{[0,s)})_{h}:=\frac{1}{C(h(t))C(h(s))}\int_{\mathbb{R}}|x|^{1-2h(x)}\hat{1\!\!1}_{[0,t)}(x)\overline{\hat{1\!\!1}}_{[0,s)}(x)\, dx,
\]
which can be extended by linearity to the pre-Hilbert space of simple
functions $(\mathcal{E}(\mathbb{R}),(\cdot,\cdot)_{h})$.
\begin{description}
\item [{(A1)}] From now on we assume that $h:[0,T]\longrightarrow(1/2,1)$
is a continuous function.\end{description}
\begin{rem}
\noindent For all $h:[0,T]\longrightarrow(1/2,1)$ satisfying (A1),
the bilinear form $(\cdot,\cdot)_{h}$ is an inner product. See \cite[Prop.~3.1]{Lebovits:2014es}.
Moreover we define the linear map $M_{h}:\mathcal{E}(\mathbb{R})\longrightarrow L^{2}(\mathbb{R}),\;1\!\!1_{[0,t)}\mapsto M_{h}1\!\!1_{[0,t)}:=M_{h(t)}1\!\!1_{[0,t)}:=M_{H}1\!\!1_{[0,t)}|_{H=h(t)}.$ \end{rem}
\begin{defn}
For $h$ satisfying (A1) we define the $L^{2}(\mu_{1})$ random variable
$B_{h}(t)$ by 
\[
B_{h}(t)=\langle\cdot,M_{h}1\!\!1_{[0,t)}\rangle.
\]

\begin{enumerate}
\item One can show that the process $(\omega,t)\mapsto B_{h}(\omega,t)$
is a one-dimensional mBm, see \cite{Lebovits:2014es}. 
\item There is a continuous version of the process $B_{h}(t)$ by Kolmogorov's
theorem and we use the same notation for this continuous version. 
\end{enumerate}
\end{defn}
\begin{rem}
The completion of $\mathcal{E}(\mathbb{R})$ with respect to $(\cdot,\cdot)_{h}$
is a Hilbert space denoted by $L_{h}^{2}(\mathbb{R})$. Moreover the
operator $M_{h}$ is an isometry between $(\mathcal{E}(\mathbb{R}),(\cdot,\cdot)_{h})$
and $(L^{2}(\mathbb{R}),(\cdot,\cdot))$, which can be extended to
an isometry between $L_{h}^{2}(\mathbb{R})$ and $L^{2}(\mathbb{R})$.
\end{rem}
The next proposition shows certain properties of 1-dimensional mBm.
\begin{prop}
The process $B_{h}(t)$, $t\geq0$ has the following properties.
\begin{enumerate}
\item The characteristic function of $B_{h}(t)$ is given by
\begin{eqnarray*}
\mathbb{E}(e^{i\lambda B_{h}(t)}) & = & \int_{S'_{1}(\mathbb{R})}e^{i\lambda\langle w,M_{h}1\!\!1_{[0,t)}\rangle}\, d\mu_{1}(w)=\exp\left(-\frac{\lambda^{2}}{2}|M1\!\!1_{[0,t)}|_{0}^{2}\right)\\
 & = & \exp\left(-\frac{\lambda^{2}}{2}t^{2h(t)}\right).
\end{eqnarray*}

\item The expectation of $B_{h}(t)$ is zero. 
\item The variance of $B_{h}(t)$ is given by
\[
\mathbb{E}(B_{h}^{2}(t))=t^{2h(t)}.
\]

\item The covariance of $B_{h}(t)$ is
\begin{eqnarray*}
R_{h}(t,s) & := & \mathbb{E}(B_{h}(t)B_{h}(s))\\
 & = & \frac{C\left(\frac{h(t)+h(s)}{2}\right)^{2}}{C(h(t))C(h(s))}\left[\frac{1}{2}(t^{h(t)+h(s)}+s^{h(t)+h(s)}-|t-s|^{h(t)+h(s)})\right].
\end{eqnarray*}

\end{enumerate}
\end{prop}
In Figures \ref{fig:fig1} and \ref{fig:fig2} the method of Wood
and Chan \cite{WoodChan} was used to simulate mBm on the interval
$[0,1]$ for different Hurst parameter functionals. 

\begin{figure}
\begin{centering}
\includegraphics[bb=0bp 150bp 612bp 655bp,clip,scale=0.5]{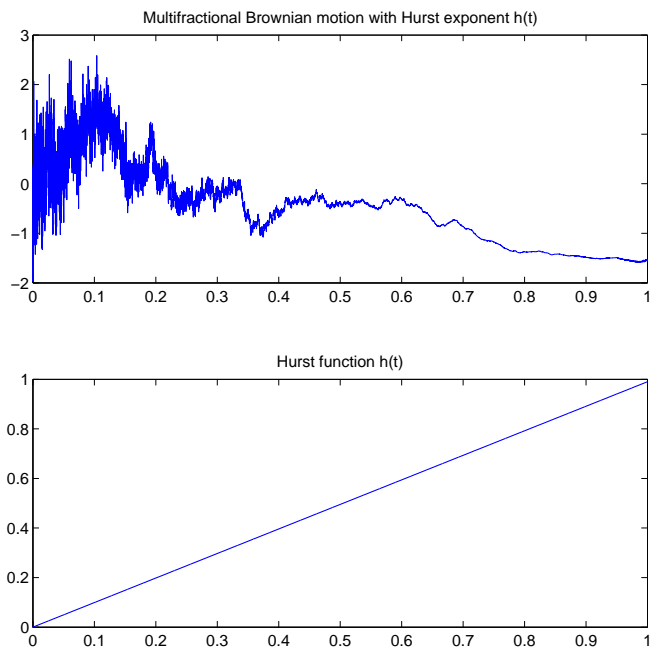}
\par\end{centering}

\noindent \protect\caption{\label{fig:fig1}Multifractional Brownian motion for parameter functional
$h.$ Simulated with the method of Wood and Chan with $s=10000$ discretization
points. Here the Hölder continuity of the path is linearly increasing
in time. }
\end{figure}

\begin{figure}
\begin{centering}
\includegraphics[bb=0bp 150bp 612bp 652bp,clip,scale=0.5]{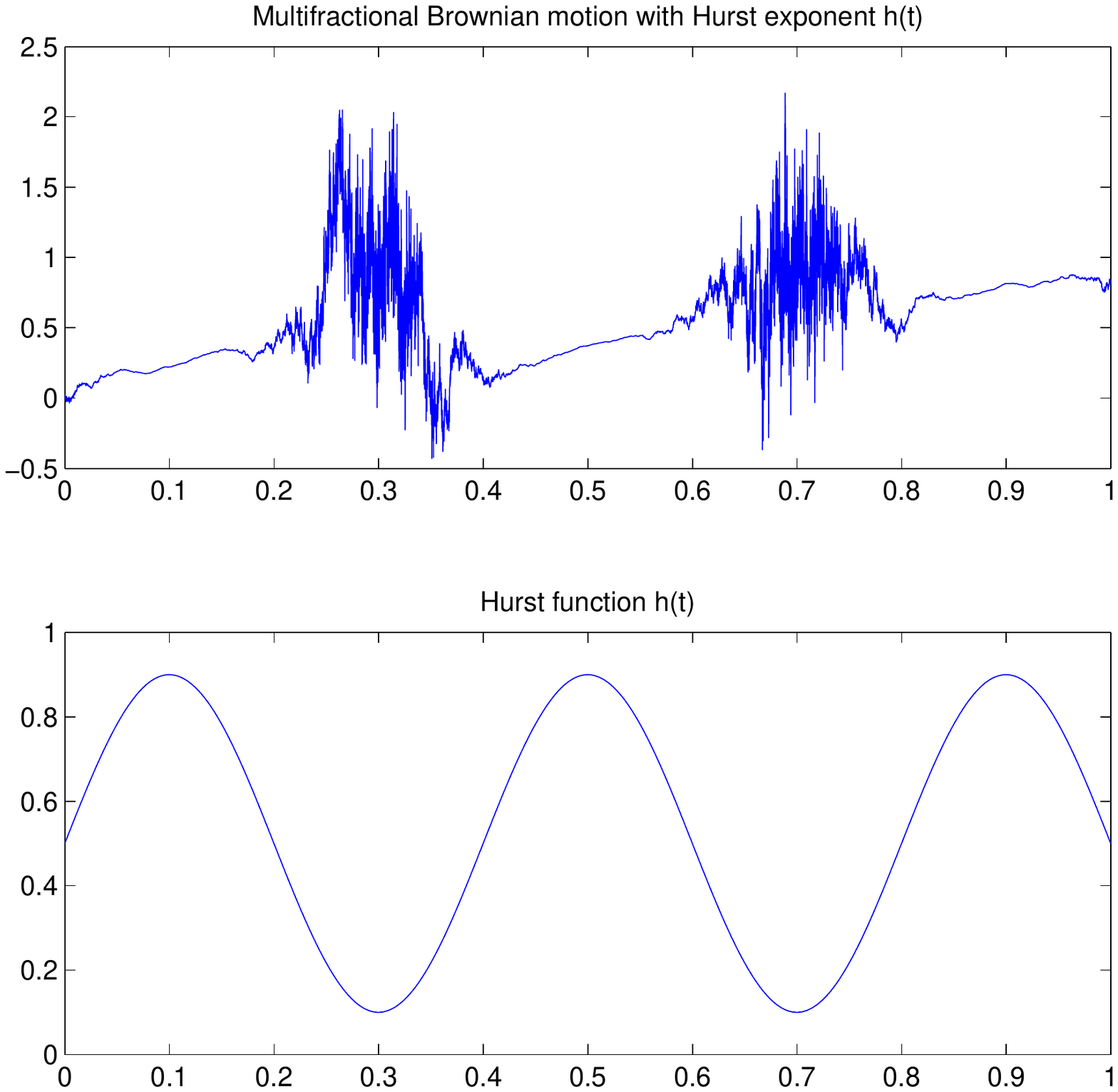}
\par\end{centering}

\protect\caption{\label{fig:fig2}Multifractional Brownian motion for parameter functional
$h$. Simulated with the method of Wood and Chan with $s=10000$ discretization
points. Here the Hölder continuity of the path is the function $h(t)=0.4+0.5\sin(5\pi t).$}

\end{figure}

Now we are ready to define the $d$-dimensional multifractional Brownian
motion.
\begin{defn}
[$d$-dimensional mBm] Let $h$ satisfy (A1). A $d$-dimensional
mBm with functional parameter $h$ is defined by
\[
\bm{B}_{h}(t)=\big(B_{h,1}(t),\ldots,B_{h,d}(t)\big),\quad t\geq0,
\]
where $B_{h,i}(t)$, $i=1,\ldots,d$, are $d$ independent $1$-dimensional
mBms.
\end{defn}
Properties of $\bm{B}_{h}(t)$:
\begin{enumerate}
\item The expectation is zero
\[
\mathbb{E}\big(\bm{B}_{h}(t)\big)=0.
\]

\item The characteristic function of $\bm{B}_{h}(t)$ is given, for any
$\bm{x}\in\mathbb{R}^{d}$, by 
\begin{eqnarray*}
\int_{S'_{d}}e^{i(\bm{\lambda},\bm{B}_{h}(\bm{w},t))_{\mathbb{R}^{d}}}\, d\mu_{d}(\bm{w}) & = & \exp\left(-\frac{1}{2}\sum_{k=1}^{d}x_{k}^{2}|M_{h}1\!\!1_{[0,t)}|_{0}^{2}\right)\\
 & = & \exp\left(-\frac{1}{2}t^{2h(t)}|\bm{x}|_{\mathbb{R}^{d}}^{2}\right).
\end{eqnarray*}

\item Covariance matrix of $\bm{B}_{h}(t)$: 
\[
\mathrm{cov}(\bm{B}_{h}(t))=(\delta_{ij}t^{2h(t)})_{i,j=1}^{d}.
\]

\end{enumerate}

\section{Local time}

The time a process spends in a certain point $\bm{y}\in\mathbb{R}^{d}$
is called the local time of the process. Formally the local time is
given by the time integral over a Dirac delta function of the process.
This Donsker's delta function is a well-defined and studied object
in white noise analysis, see e.g. \cite{LLSW93,HKPS93,O94,Kuo96}.
For mBm recently the concept of local times was introduced and studied,
see e.g. \cite{boufoussi2007,meerschaert2008,ayache2011} and references
therein.

In this section we determine the kernel functions for the local time
and the regularized local time of mBm. Moreover we show that for a
suitable number of truncated kernel functions the local time of mBm
is a Hida distribution and can be obtained as a limit of the regularized
local time.
\begin{prop}
For $t>0$ the Bochner integral
\begin{equation}
\delta(\bm{B}_{h}(t)):=\left(\frac{1}{2\pi}\right)^{d}\int_{\mathbb{R}^{d}}e^{i(\bm{\lambda},\bm{B}_{h}(t))_{\mathbb{R}^{d}}}\, d\bm{\lambda}\label{eq:Donsker-delta}
\end{equation}
is a Hida distribution and for any $\boldsymbol{\varphi}\in S_{d}$
its $S$-transform given by 
\begin{equation}
S\delta(\bm{B}_{h}(t))(\bm{\varphi})=\left(\frac{1}{\sqrt{2\pi}\, t^{h(t)}}\right)^{d}\exp\left(-\frac{1}{2t^{2h(t)}}\left|\int_{\mathbb{R}}\bm{\varphi}(x)(M_{h}1\!\!1_{[0,t)})(x)\, dx\right|_{\mathbb{R}^{d}}^{2}\right).\label{eq:S-transf-delta}
\end{equation}
\end{prop}
\begin{proof}
First we compute the $S$-transform of the integrand in (\ref{eq:Donsker-delta})
for any $\boldsymbol{{\varphi}}\in S_{d}$:
\begin{equation}
Se^{i(\bm{\lambda},\bm{B}_{h}(t))_{\mathbb{R}^{d}}}(\boldsymbol{{\varphi}})=\exp\left(-\frac{1}{2}|\bm{\lambda}|_{\mathbb{R}^{d}}^{2}t^{2h(t)}+i\Big(\bm{\lambda},\int_{\mathbb{R}}\boldsymbol{{\varphi}}(x)(M_{h}1\!\!1_{[0,t)})(x)\, dx\Big)_{\mathbb{R}^{d}}\right).\label{eq:S-transf-delta-integ}
\end{equation}
It is clear that the $S$-transform is $\bm{\lambda}$-measurable
for any $\bm{{\varphi}}\in S_{d}$. On the other hand, for any $z\in\mathbb{C}$
and all $\bm{\varphi}\in S_{d}$ we obtain
\begin{eqnarray*}
 &  & |Se^{i(\bm{\lambda},\bm{B}_{h}(t))_{\mathbb{R}^{d}}}(z\bm{\varphi})|\\
 & = & \left|\exp\left(-\frac{1}{2}|\bm{\lambda}|_{\mathbb{R}^{d}}^{2}t^{2h(t)}+iz\Big(\bm{\lambda},\int_{\mathbb{R}}\bm{\varphi}(x)(M_{h}1\!\!1_{[0,t)})(x)\, dx\Big)_{\mathbb{R}^{d}}\right)\right|\\
 & \leq & \exp\left(-\frac{1}{4}|\bm{\lambda}|_{\mathbb{R}^{d}}^{2}t^{2h(t)}\right)\exp\left(-\frac{1}{4}|\bm{\lambda}|_{\mathbb{R}^{d}}^{2}t^{2h(t)}+|z|\left|\Big(\bm{\lambda},\int_{\mathbb{R}}\bm{\varphi}(x)(M_{h}1\!\!1_{[0,t)})(x)\, dx\Big)_{\mathbb{R}^{d}}\right|\right)\\
 & \leq & \exp\left(-\frac{1}{4}|\bm{\lambda}|_{\mathbb{R}^{d}}^{2}t^{2h(t)}\right)\exp\left(-\left(\frac{1}{2}|\bm{\lambda}|_{\mathbb{R}^{d}}t^{h(t)}-\frac{1}{t^{h(t)}}|z|\left|\int_{\mathbb{R}}\bm{\varphi}(x)(M_{h}1\!\!1_{[0,t)})(x)\, dx\right|\right)^{2}\right)\\
 &  & \times\exp\left(\frac{1}{t^{2h(t)}}|z|^{2}\left|\int_{\mathbb{R}}\bm{\varphi}(x)(M_{h}1\!\!1_{[0,t)})(x)\, dx\right|_{\mathbb{R}^{d}}^{2}\right)\\
 & \leq & \exp\left(-\frac{1}{4}|\bm{\lambda}|_{\mathbb{R}^{d}}^{2}t^{2h(t)}\right)\exp\left(\frac{1}{t^{2h(t)}}|z|^{2}\left|\int_{\mathbb{R}}\bm{\varphi}(x)(M_{h}1\!\!1_{[0,t)})(x)\, dx\right|_{\mathbb{R}^{d}}^{2}\right)\\
 & \leq & \exp\left(-\frac{1}{4}|\bm{\lambda}|_{\mathbb{R}^{d}}^{2}t^{2h(t)}\right)\exp\left(|z|^{2}|\bm{\varphi}|{}_{0}^{2}\right).
\end{eqnarray*}
Thus we have the following bound
\[
|Se^{i(\bm{\lambda},\bm{B}_{h}(t))_{\mathbb{R}^{d}}}(z\bm{\varphi})|\leq\exp\left(-\frac{1}{4}|\bm{\lambda}|_{\mathbb{R}^{d}}^{2}t^{2h(t)}\right)\exp\left(|z|^{2}|\bm{\varphi}|_{0}^{2}\right),
\]
where, as a function of $\bm{\lambda}$, the first factor is integrable
on $\mathbb{R}^{d}$ and the second factor is constant. The result
(\ref{eq:S-transf-delta}) follows from (\ref{eq:S-transf-delta-integ})
and integration with respect to $\bm{\lambda}$. 
\end{proof}
In the following theorem we characterize the truncated local time
of mBm as a Hida distribution. Therefore we use the notation $\delta^{(N)}$
for the truncated Donsker's delta function as in \ref{def:trunc_kern},
i.e.~for any $\bm{\varphi}\in S_{d}$

\[
S\delta^{(N)}(\bm{B}_{h}(t))(\bm{\varphi})=\left(\frac{1}{\sqrt{2\pi}\, t^{h(t)}}\right)^{d}\exp_{N}\left(-\frac{1}{2t^{2h(t)}}\left|\int_{\mathbb{R}}\bm{\varphi}(x)(M_{h}1\!\!1_{[0,t)})(x)\, dx\right|_{\mathbb{R}^{d}}^{2}\right),
\]

with $\exp_{N}(x):=\sum_{n=N}^{\infty}\frac{{1}}{n!}x^{n}$ for $x\in\mathbb{C}.$
\begin{description}
\item [{(A2)}] Let $h$ satisfy (A1) and $\sup_{t\in[0,T)}h(t)<\frac{1+2N}{2N+d}$,
for a fixed $N\in\mathbb{N}_{0}$ and $d\in\mathbb{N}$.\end{description}
\begin{thm}
\label{thm:LT_Hida_distrib}For $h$ satisfying (A2) with dimension
$d\in\mathbb{N}$ and $N\in\mathbb{{N}}_{0}$ , the Bochner integral
\[
L_{h}^{(N)}(T):=\int_{0}^{T}\delta^{(N)}(\bm{B}_{h}(t))\, dt
\]
is a Hida distribution.\end{thm}
\begin{proof}
The proof uses again Corollary~\ref{cor:hida_integral} with respect
to the Lebesgue measure in $[0,T]$. It follows from (\ref{eq:S-transf-delta})
that
\begin{equation}
S\delta^{(N)}(\bm{B}_{h}(t))(\bm{\varphi})=\left(\frac{1}{\sqrt{2\pi}\, t^{h(t)}}\right)^{d}\exp_{N}\left(-\frac{1}{2t^{2h(t)}}\left|\int_{\mathbb{R}}\bm{\varphi}(x)(M_{h}1\!\!1_{[0,t)})(x)\, dx\right|_{\mathbb{R}^{d}}^{2}\right),\label{eq:S-transf-trunc-delta}
\end{equation}
which is measurable in $t$ for every $\bm{\varphi}\in S_{d}$. Using
Proposition~(\ref{prop:Properties_M_H})-4.~we obtain the following
bound for any $z\in\mathbb{C}$ and all $\bm{\varphi}\in S_{d}$ 
\begin{eqnarray*}
 &  & |S\delta^{(N)}(\bm{B}_{h}(t))(z\bm{\varphi})|\\
 & \leq & \left(\frac{1}{\sqrt{2\pi}\, t^{h(t)}}\right)^{d}\exp_{N}\left(\frac{1}{2t^{2h(t)}}|\gamma(h(t))|^{2}\,|\Theta_{h(t)}|_{-p}^{2}t^{2}|z|^{2}\left(\sup_{x\in\mathbb{R}}|\varphi(x+\cdot)|_{p}\right)^{2}\right).\\
 & \leq & \left(\frac{1}{\sqrt{2\pi}\, t^{h(t)}}\right)^{d}\exp_{N}\left(\frac{1}{2}K(h)t^{2-2h(t)}|z|^{2}\|\bm{\varphi}\|^{2}\right),
\end{eqnarray*}
where $K(h)$ is independent of $t$ (note that $\sup_{t\in[0,T]}|h(t)|=\beta<1$)
and we defined the continuous norm on $S_{d}$ by 
\[
\|\bm{\varphi}\|:=\sup_{x\in\mathbb{R}}|\varphi(x+\cdot)|_{p}.
\]
The estimation 
\[
\exp_{N}\left(\frac{1}{2}K(h)t^{2-2h(t)}|z|^{2}\|\bm{\varphi}\|^{2}\right)\leq t^{2N(1-h(t))}\exp\left(\frac{K(h)}{2}|z|^{2}\|\bm{\varphi}\|^{2}\right)
\]
allows us to obtain the bound 
\begin{eqnarray*}
|S\delta^{(N)}(\bm{B}_{h}(t))(z\bm{\varphi})| & \leq & \left(\frac{1}{\sqrt{2\pi}}\right)^{d}t^{2N(1-h(t))-dh(t)}\exp\left(\frac{K(h)}{2}|z|^{2}\|\bm{\varphi}\|^{2}\right),
\end{eqnarray*}
which is integrable in $t\in[0,T)$ due to (A2). The proof follows
from the application of Corollary~\ref{cor:hida_integral}.\end{proof}
\begin{thm}
For $h$ satisfying (A2) for dimension $d$ and $N\in\mathbb{{N}}_{0}$,
the kernels functions of $L_{h}^{(N)}(T)$ are given by
\begin{equation}
F_{h,2\bm{n}}(u_{1},\ldots,u_{2n})=\frac{1}{\bm{n}!}\left(\frac{1}{\sqrt{2\pi}}\right)^{d}\left(-\frac{1}{2}\right)^{n}\int_{0}^{T}\frac{1}{t^{2h(t)N+dh(t)}}\prod_{j=1}^{2n}(M_{h(t)}1\!\!1_{[0,t]})(u_{j})\, dt\label{eq:kernels_L-trunc}
\end{equation}
for each $\bm{n}\in\mathbb{N}^{d}$ such that $n\geq N$. All the
other kernels $F_{h,\bm{n}}$ are zero.\end{thm}
\begin{proof}
The kernels of $L_{h}^{(N)}(T)$ are obtained by its $S$-transform.
Therefore we use Corollary~\ref{cor:hida_integral} and integrate
(\ref{eq:S-transf-trunc-delta}) over $[0,T]$. For any $\bm{\varphi}\in S_{d}$,
we have
\begin{eqnarray*}
SL_{h}^{(N)}(T)(\bm{\varphi}) & = & \left(\frac{1}{\sqrt{2\pi}}\right)^{d}\int_{0}^{T}\frac{1}{t^{dh(t)}}\sum_{n=N}^{\infty}\frac{(-1)^{n}}{2^{n}t^{2nh(t)}}\\
 &  & \times\sum_{{n_{1},\ldots,n_{d}\in\mathbb{N}\atop n_{1}+\ldots+n_{d}=n}}\frac{1}{n_{1}!\ldots n_{d}!}\prod_{j=1}^{d}\left(\int_{\mathbb{R}}\varphi_{j}(x)(M_{h}1\!\!1_{[0,t)})(x)\, dx\right)^{2n_{j}}dt\\
 & = & \left(\frac{1}{\sqrt{2\pi}}\right)^{d}\int_{0}^{T}\sum_{n=N}^{\infty}\left(-\frac{1}{2}\right)^{n}\frac{1}{t^{2nh(t)+dh(t)}}\\
 &  & \times\sum_{{n_{1},\ldots,n_{d}\in\mathbb{N}\atop n_{1}+\ldots+n_{d}=n}}\frac{1}{\bm{n}!}\prod_{j=1}^{d}\left(\int_{\mathbb{R}}\varphi_{j}(x)(M_{h}1\!\!1_{[0,t)})(x)\, dx\right)^{2n_{j}}dt.
\end{eqnarray*}
Comparing it with the general form of the chaos expansion 
\[
L_{h}^{(N)}(T)=\sum_{\bm{n}\in\mathbb{N}^{d}}\langle:\bm{w}^{\otimes\bm{n}}:,F_{h,\bm{n}}\rangle
\]
we obtain $F_{h,\bm{n}}$ as in (\ref{eq:kernels_L-trunc}). This
completes the proof. \end{proof}
\begin{rem}
The result of Theorem~\ref{thm:LT_Hida_distrib} shows that for $d=1$
all local times are well-defined, but for $d\geq2$ they are well-defined
after truncation of divergent terms. This motivates the study of a
regularized version, namely we discuss 
\[
L_{h,\varepsilon}(T):=\int_{0}^{T}\delta_{\varepsilon}(\bm{B}_{h}(t))\, dt,\quad\varepsilon>0,
\]
where 
\[
\delta_{\varepsilon}(\bm{B}_{h}(t)):=\left(\frac{1}{\sqrt{2\pi\varepsilon}}\right)^{d}\exp\left(-\frac{1}{2\varepsilon}|\bm{B}_{h}(t)|_{\mathbb{R}^{d}}^{2}\right).
\]
\end{rem}
\begin{thm}
Let $\varepsilon>0$ be given and $h$ satisfy (A2). 
\begin{enumerate}
\item The functional $L_{h,\varepsilon}(T)$ is a Hida distribution with
kernels functions given by
\begin{eqnarray}
F_{h,\varepsilon,2\bm{n}}(u_{1},\ldots,u_{2n}) & = & \frac{1}{\bm{n}!}\left(\frac{1}{\sqrt{2\pi}}\right)^{d}\left(-\frac{1}{2}\right)^{n}\int_{0}^{T}\frac{1}{(\varepsilon+2h(t))^{n+d/2}}\label{eq:chaos_renor_LT}\\
 &  & \times\prod_{j=1}^{2n}(M_{h}1\!\!1_{[0,t)})(u_{j})\, dt
\end{eqnarray}
for each $\bm{n}=(n_{1},\ldots,n_{d})\in\mathbb{N}^{d}$ and $F_{h,\varepsilon,\bm{n}}=0$
if at least one of the $n_{j}$ is an odd number. 
\item For $\varepsilon$ tends to zero the truncated functional $L_{h,\varepsilon}^{(N)}(T)$
converges strongly in $(S_{d})'$ to the truncated local time $L_{h}^{(N)}(T)$. 
\end{enumerate}
\end{thm}
\begin{proof}
1. First we compute the $S$-transform of the integrand of $L_{h,\varepsilon}(T)$.
For any $\bm{\varphi}\in S_{d},$ we obtain
\begin{eqnarray*}
S\delta_{\varepsilon}(\bm{B}_{h}(t))(\bm{\varphi}) & = & \left(\frac{1}{\sqrt{2\pi(\varepsilon+t^{2h(t)})}}\right)^{d}\exp\bigg(-\frac{1}{2(\varepsilon+t^{2h(t)})}\\
 &  & \times\left|\int_{\mathbb{R}}\bm{\varphi}(x)(M_{h}1\!\!1_{[0,t)})(x)\, dx\right|_{\mathbb{R}^{d}}^{2}\bigg),
\end{eqnarray*}
which is measurable in $t$. Thus, for any $z\in\mathbb{C}$ and $\bm{\varphi}\in S_{d}$,
by Proposition~\ref{prop:Properties_M_H}-4. we arrive at the following
bound 
\[
|S\delta_{\varepsilon}(\bm{B}_{h}(t))(z\bm{\varphi})|\leq\left(\frac{1}{\sqrt{2\pi(\varepsilon+t^{2h(t)})}}\right)^{d}\exp\left(K(h)|z|^{2}\frac{t^{2}}{2(\varepsilon+t^{2h(t)})}\|\bm{\varphi}\|^{2}\right).
\]
On the other hand, $\frac{t^{2}}{2(\varepsilon+t^{2h(t)})}$ is bounded
in $[0,T]$ and $(\varepsilon+t^{2h(t)})^{-d/2}$ is integrable on
$[0,T]$, therefore we may conclude, by Corollary~\ref{cor:hida_integral},
that $L_{h,\varepsilon}(T)\in(S_{d})'$. In addition, for any $\bm{\varphi}\in S_{d}$,
we have 
\begin{eqnarray*}
\big(SL_{h,\varepsilon}(T)\big)(\bm{\varphi}) & = & \int_{0}^{T}\big(S\delta_{\varepsilon}(\bm{B}_{h}(t))\big)(\bm{\varphi})\, dt\\
 & = & \left(\frac{1}{\sqrt{2\pi}}\right)^{d}\int_{0}^{T}\frac{1}{(\varepsilon+t^{2h(t)})^{d/2}}\sum_{n=0}^{\infty}\frac{(-1)^{n}}{2^{n}(\varepsilon+t^{2h(t)})^{n}}\\
 &  & \times\sum_{{n_{1},\ldots,n_{d}\in\mathbb{N}\atop n_{1}+\ldots+n_{d}=n}}\frac{1}{n_{1}!\ldots n_{d}!}\prod_{j=1}^{d}\left(\int_{\mathbb{R}}\varphi_{j}(x)(M_{h}1\!\!1_{[0,t)})(x)\, dx\right)^{2n_{j}}dt\\
 & = & \left(\frac{1}{\sqrt{2\pi}}\right)^{d}\int_{0}^{T}\sum_{n=0}^{\infty}\left(-\frac{1}{2}\right)^{n}\frac{1}{(\varepsilon+t^{2h(t)})^{n+d/2}}\\
 &  & \times\sum_{{n_{1},\ldots,n_{d}\in\mathbb{N}\atop n_{1}+\ldots+n_{d}=n}}\frac{1}{\bm{n}!}\prod_{j=1}^{d}\left(\int_{\mathbb{R}}\varphi_{j}(x)(M_{h}1\!\!1_{[0,t)})(x)\, dx\right)^{2n_{j}}dt.
\end{eqnarray*}
Comparing the latter expression with the kernels $F_{h,\varepsilon,\bm{n}}$
from the chaos expansion of $L_{h,\varepsilon}(T)$
\[
L_{h,\varepsilon}(T)=\sum_{\bm{n}\in\mathbb{N}^{d}}\langle:\bm{w}^{\otimes\bm{n}}:,F_{h,\varepsilon,\bm{n}}\rangle,
\]
we conclude that whenever one of the $n_{j}$ in $\bm{n}=(n_{1},\ldots,n_{d})$
is odd we have $F_{h,\varepsilon,\bm{n}}=0$, otherwise they are given
by the expression (\ref{eq:chaos_renor_LT}). 

\noindent 2. To check the convergence we shall use Corollary~\ref{cor:conv-hida-seq}
and fact that
\[
\big(SL_{h,\varepsilon}^{(N)}(T)\big)(\bm{\varphi})=\int_{0}^{T}\big(S\delta_{\varepsilon}(\bm{B}_{h}(t))\big)(\bm{\varphi})\, dt.
\]
Thus, for all $z\in\mathbb{C}$ and all $\bm{\varphi}\in S_{d}$ we
estimate $\big(SL_{h,\varepsilon}^{(N)}(T)\big)(z\bm{\varphi})$ by
\begin{eqnarray*}
\big|\big(SL_{h,\varepsilon}^{(N)}(T)\big)(z\bm{\varphi})\big| & \leq & \int_{0}^{T}\big|\big(S\delta_{\varepsilon}(\bm{B}_{h}(t))\big)(\bm{\varphi})\big|\, dt\\
 & \leq & \left(\frac{1}{\sqrt{2\pi\varepsilon}}\right)^{d}\int_{0}^{T}\exp\left(\frac{K(h)}{2\varepsilon}|z|^{2}t^{2}\|\bm{\varphi}\|^{2}\right)dt\\
 & \leq & \left(\frac{1}{\sqrt{2\pi\varepsilon}}\right)^{d}\exp\left(\frac{C(h,T)}{2\varepsilon}|z|^{2}\|\bm{\varphi}\|^{2}\right),
\end{eqnarray*}
for a certain constant $C(h,T)>0$, which shows the uniform boundedness
condition. Moreover, using similar calculations as in Theorem~\ref{thm:LT_Hida_distrib},
for any $t\in[0,T]$, yields
\begin{eqnarray*}
\big|\big(SL_{h,\varepsilon}^{(N)}(T)\big)(z\bm{\varphi})\big| & \leq & \left(\frac{1}{\sqrt{2\pi}t^{h(t)}}\right)^{d}\exp_{N}\left(\frac{K(h)}{2}t^{2-2h(t)}\|\bm{\varphi}\|^{2}\right)\\
 & \leq & \left(\frac{1}{\sqrt{2\pi}}\right)^{d}t^{2N(1-h(t))-dh(t)}\exp\left(\frac{K(h)}{2}\|\bm{\varphi}\|^{2}\right).
\end{eqnarray*}
This upper bound together with the fact that $1/2<h(t)<1$, for any
$t\in[0,T]$, gives an integrable function on $[0,T]$. Finally, an
application of Lebesgue's dominated convergence theorem implies the
other condition in order to apply Corollary~\ref{cor:conv-hida-seq}.
This completes the proof.
\end{proof}

\subsection*{Acknowledgments}

H.P.S.~and W.B.~would like to thank the financial support of FCT
\textendash{} Fundação para a Ciência e a Tecnologia through the project
Refª PEst-OE/MAT/UI0219/2014. W.B. thanks for the fellowship in the
FCT-project PTDC/MAT-STA/1284/2012.

\bibliographystyle{alpha}
\bibliography{luis}

\end{document}